\documentclass[12pt]{article}

\usepackage{amssymb}
\usepackage{amsmath}
\usepackage{enumerate}
\usepackage{graphicx}
\usepackage{amscd}
\usepackage{latexsym}
\usepackage{epsfig}
\usepackage[latin1]{inputenc}
\usepackage{a4wide}
\usepackage[dvipsnames,usenames]{color}

\newtheorem{remark}{Remark}[section]
\newtheorem{theorem}{Theorem}[section]
\newtheorem{proposition}{Proposition}[section]
\newtheorem{corollary}{Corollary}[section]
\newtheorem{assertion}{Claim}[section]
\newtheorem{definition}{Definition}[section]

\newenvironment{proof}{\trivlist
\item[\hskip\labelsep{\em Proof}\,:]}{\hfill{$\Box$}\endtrivlist}

\def\r{\mathbb{R}} \def\R{\mathbb{R}}
\def\n{\mathbb{N}} \def\N{\mathbb{N}}
\def\c{\mathbb{C}} \def\C{\mathbb{C}}
\def\h{\mathbb{H}} 
\def\l{\mathbb{L}} \def\L{\mathbb{L}}
\def\z{\mathbb{Z}} 
\def\S{\mathbb{S}}

\def\Nb{\mathcal{N}}\def\nb{\mathcal{N}}

\def\Sb{\mathcal{S}}
\def\Tb{\mathcal{T}}

\def\Db{\mathcal{D}}

\def\div{{\rm{Div}}}

\newcommand{\ovr}{\overline}

\title{The number of conformally equivalent maximal graphs}
\author{Isabel Fernández
\thanks{Research partially supported by
Spanish MEC-FEDER Grant MTM2007-64504, and Regional J. Andalucía Grants P06-FQM-01642 and FQM325.\newline
}}

\begin{document}
\maketitle

\begin{abstract}
We show that the number of entire maximal graphs with finitely
many singular points that are conformally equivalent
is a universal constant that depends only on the number of singularities,
namely $2^n$ for graphs with $n+1$ singularities.
We also give an explicit description of the family of entire maximal graphs with a finite number
of singularities all of them lying on a plane orthogonal to the limit normal vector at infinity.
\end{abstract}

%%%%%%%%%%%%%%%%%%%%%%%%%%%%%%%%%%%%%%%%%%%%%%%%%%%%%%%%%%%%%%%%
\section{Introduction}\label{sec:intro}

The present paper is devoted to the study of maximal graphs in the Lorentz-Minkowski space $\L^3=(\R^3,\langle\cdot,\cdot\rangle),$ where  $\langle(x_1,x_2,x_3),(y_1,y_2,y_3)\rangle=x_1y_1+x_2y_2-x_3y_3$. Maximal graphs appear in a natural way when considering variational problems. If $u:\Omega\subset\R^2\equiv\{x_3=0\}\to\R$ is a smooth function defining a spacelike graph in $\L^3$ (that is, a graph with Riemannian induced metric), then its area is given by the expression
$$A(u)=\int_\Omega\sqrt{1-|\nabla u|^2},$$
(recall that $|\nabla u|<1$ since the graph is spacelike).
The corresponding equation for the critical points of the area functional in $\L^3$ is
\begin{equation}\label{eq:maximal}
\mbox{Div}\frac{\nabla u}{\sqrt{1-|\nabla u|^2}} =0.
\end{equation}
Spacelike graphs satisfying this (elliptic) differential equation are called {\em maximal
graphs}, since they represent local maxima for the area functional.
Geometrically, this condition is equivalent to the fact that the
mean curvature of the surface in $\L^3$ vanishes identically.
Besides of their ma\-the\-ma\-ti\-cal interest, these surfaces, and
more generally those having
constant mean curvature, have a significant importance in physics \cite{mardsen}.\\

\begin{figure}[htbp]\centering
\includegraphics[width=.8\textwidth]{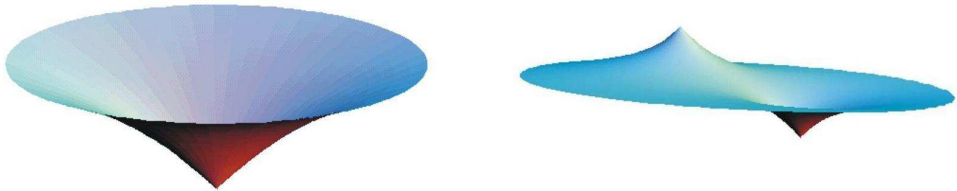}\label{fig:catrie}\caption{Left: Lorentzian catenoid. Right: Riemann type surface.}
\end{figure}

From a global point of view, it is known by Calabi's theorem \cite{calabi} that the only everywhere regular complete maximal
surface is the plane. In particular, there are no entire maximal graphs besides the trivial one.
This motivates to allow the existence of singularities, i.e., points of the surface where the metric degenerates.
We will focus here our attention to the case where the singular set is the {\em smallest} possible, that is,
a finite number of points. The first and most known example is the Lorentzian catenoid (Figure \ref{fig:catrie}, left),
an entire maximal graph with one singular point, and actually the only one as proved in \cite{ecker},
but there are examples with any arbitrary number of singularities.
Among them it is worth mentioning the Riemann type maximal graphs (Figure \ref{fig:catrie}, right)
obtained in \cite{lopez-lopez-souam}, with two singular points and characterized by the property of being
foliated by circles and lines. Other highly symmetric examples with arbitrary number of singularities (even infinitely many)
were constructed in \cite{praga} (Figure \ref{fig:3y4picos}). Actually there is a huge amount of such graphs.
Indeed, in \cite{conelike} the authors study  the moduli space $\mathcal{G}_n$  of entire maximal graphs with $n+1$ singularities,
proving that it is an analytic manifold of dimension $3n+4$. A global system of coordinates in this space is given
by the position of the singular points in $\L^3$ and a real number called {\em the logarithmic growth} that controls
the asymptotic behavior.\\

\begin{figure}[htbp]\centering
\includegraphics[width=\textwidth]{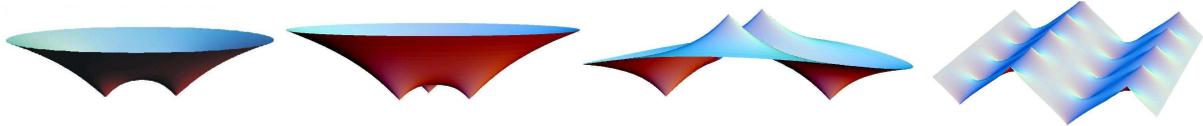}\caption{Entire maximal graphs with isolated singularities.}\label{fig:3y4picos}
\end{figure}

If $u:\Omega\to\R$ defines a maximal graph, singular points appear where $|\nabla u|=1$.
At a singular point, the PDE \eqref{eq:maximal} stops being elliptic. Moreover, the tangent plane of the surface becomes lightlike,
the normal vector has no well defined limit, and the surface is asymptotic to a half of the light cone of the singular point.
For this reason they are called {\em conelike singularities}. It should be pointed out that a maximal surface
with isolated conelike singularities is an entire graph if and only if it is complete (that is, divergent curves
have infinite length), as proved in \cite{conelike}.\\

If $S$ is a maximal surface with singular set $F\subset S$, its
regular part $S\setminus F$ has a natural conformal structure
associated to its Riemannian metric. The conformal type of a
maximal surface has been widely studied, for example in \cite{paralore,alias}
parabolicity criteria for maximal surfaces are given, but there also
exist hyperbolic examples, \cite{alarcon1,alarcon2,muy}.

In the case of entire graphs with $n+1$ singularities,
it turns out that $S\setminus F$ is conformally equivalent to a $n$-connected
circular domain of the complex plane, that is, the plane with
$n+1$ discs removed. Each one of these boundary circles
corresponds to a singular point of the graph. Our aim in this paper is to study the space of
entire maximal graphs with the same conformal structure, that is
\begin{quote}
{\bf Problem.} {\em Given a $n$-connected circular domain $\Omega$ of the complex plane, how many entire maximal graphs with $n+1$ singularities
are there whose conformal structure is biholomorphic to $\Omega$?}
\end{quote}

We will answer this question by proving that the number of (non congruent) maximal graphs supported by a fixed
circular domain is finite and does not depend on the circular domain, but only on the number
of connected component of the boundary, that is, the number of singularities. This will be the aim
of Section \ref{sec:first}. Thus, our problem reduces to compute the number of graphs for a fixed conformal structure.
In Section \ref{sec:count} we will fix an specific $n$-connected circular domain (Definition \ref{def:omega})
and we will find out how many entire graphs are there with this conformal structure, obtaining that there are exactly $2^n$ non-congruent surfaces.
%As a consequence, we will also obtain that the moduli space $\hat\mathcal{G}_n$ of entire maximal graphs with $n+1$ singularities up to congruences
%has exactly $2^n$ connected components (Theorem \ref{th:connected}).
Moreover, the graphs constructed in Section \ref{sec:count} can be characterized by the property of
having all their singularities in a plane orthogonal to the limit normal vector at infinity (Theorem \ref{th:alineados}).

Let us point out that our main result contrast with the analogous
problem in the related theory of solutions to the Monge-Ampère
equation
\begin{equation}\label{eq:hess} \mbox{Hess}(u) =1.\end{equation}
Specifically, in \cite{GMM} it is proved that any solution to \eqref{eq:hess} globally defined on $\R^2$ with
finitely many isolated singularities is uniquely determined by its
associated conformal structure, which is also a circular domain of the
complex plane.

%%%%%%%%%%%%%%%%%%%%%%%%%%%%%%%%%%%%%%%%%%%%%%%%%%%%%%%%%%%
\section{Preliminaries}\label{sec:prelim}

\subsection{Maximal surfaces}\label{sec:maximal}

A differentiable immersion $X:M\to\l^3$ from a surface $M$ to
$\l^3$ is said to be spacelike if the tangent plane at any point
is spacelike, that is to say, the induced metric on $M$ is
Riemannian. The Gauss map of a spacelike surface in $\L^3$ takes values in the sphere of radius $-1$, $\h^2=\{p\in\L^3\;:\;\langle p,p\rangle=-1\}$.
Since $\h^2$ has two connected components, $\h^2_+=\h^2\cap\{x_3>0\}$ and $\h^2_-=\h^2\cap\{x_3<0\}$, spacelike surfaces are always orientable.\\

A maximal immersion is a spacelike immersion whose mean curvature
vanishes. A remarkable property of maximal surfaces in $\l^3$ is the existence of a Weierstrass-type representation for maximal surfaces, similar to the one of minimal surfaces. Roughly speaking, the Weierstrass representation of a conformal maximal immersion $X:M\to\l^3$ is a pair $(g,\phi_3)$ of a meromorphic function and a holomorphic $1$-form defined on $M$ such that, up to translation, the immersion can be recovered as
 \begin{equation}\label{eq:repr}
X(p):=\mbox{Real} \int_{p_0}^p \big(
\frac{i}{2}(\frac{1}{g}-g)\phi_3,
\frac{-1}{2}(\frac{1}{g}+g)\phi_3,
\phi_3
 \big) ,
 \end{equation}
where $p_0\in M$ is an arbitrary point. It is worth mentioning that $g$ agrees with the stereographic projection of the Gauss map of the surface.
We refer to \cite{kobayashi, ecker} and Theorem \ref{th:representation} below for more details.

We will focus our attention to entire maximal graphs, that is, maximal graphs defined on the whole plane $\{x_3=0\}$.
As we explained in Section \ref{sec:intro}, the only everywhere regular example is the plane \cite{calabi},
and so singularities (i.e., points where the induced metric converges to zero) appear in a natural way in this setting.
The following theorem condense the information regarding the global structure of entire maximal graphs with isolated singularities
(also called {\em conelike singularities}).
\begin{proposition}[Global behavior, \cite{conelike}]  \label{pro:global}
Let $S$ be a surface with isolated singularities in $\L^3$. Then the following two statements are equivalents:
\begin{enumerate}[(i)]
\item $S$ is a complete embedded maximal surface,
\item $S$ is an entire graph over any spacelike plane.
\end{enumerate}

In  this case $S$ is asymptotic at infinity to either a half-catenoid or a plane.
If we label $F\subset S$ as the singular set,  $S\setminus F$ is conformally equivalent to $\Omega_0:=\c\setminus\cup_{p\in F} D_p,$
where $D_p$ are pairwise disjoint closed discs.
Moreover, the associated conformal
reparameterization $X:\Omega_0\to\L^3$  extends analytically to
$\Omega:=\c\setminus\cup_{p\in F}\mbox{Int}(D_p)$ by putting
$X(\partial(D_p))=X(p).$ The point $p_\infty=\infty$ is called the {\em
end} of the surface.
\end{proposition}
%%%%%%%%%%%%%%%%%%%%%%%%%%%%%%%%%%%%%%%%%%%%%%%%%%%%%%%%%%%%%%%%%%%%%%

\subsection{Double surface and representation theorem}\label{sec:repr}

As showed in the previous section, the underlying conformal structure of an entire maximal graph with an isolated set
of singularities is conformally equivalent to a circular domain in the complex plane. We now go into this aspect in depth
to obtain a representation theorem for entire maximal graphs with a finite number of singularities that will be crucial in our study.

For any finitely connected circular domain $\Omega=\c\setminus\cup_{j=1}^k
\mbox{Int}(D_j)$, let $\Omega^*$ be its mirror surface and $\Nb$ the double surface obtained by gluing $\Omega$ and
$\Omega^*$ along their common boundaries as in Figure \ref{fig:mirror} (see \cite{farkas} for an explicit description of this construction).
It is clear that $\Nb$ is a compact Riemann surface
 of genus $k-1$ minus two points. We denote by $\ovr{\Nb}$ the
compactification of $\Nb$ by adding these two points.

Finally, we label $J:\Nb\to\Nb$ as the mirror involution mapping a
point in $\Omega$ into its mirror image and viceversa. Notice that $J$
extends to an antiholomorphic involution on $\ovr{\Nb}$, and its
fixed point set of $J$ coincides with
$\partial\Omega\equiv\partial\Omega^*.$

\begin{figure}[htbp]
\centering
\includegraphics[width=16cm]{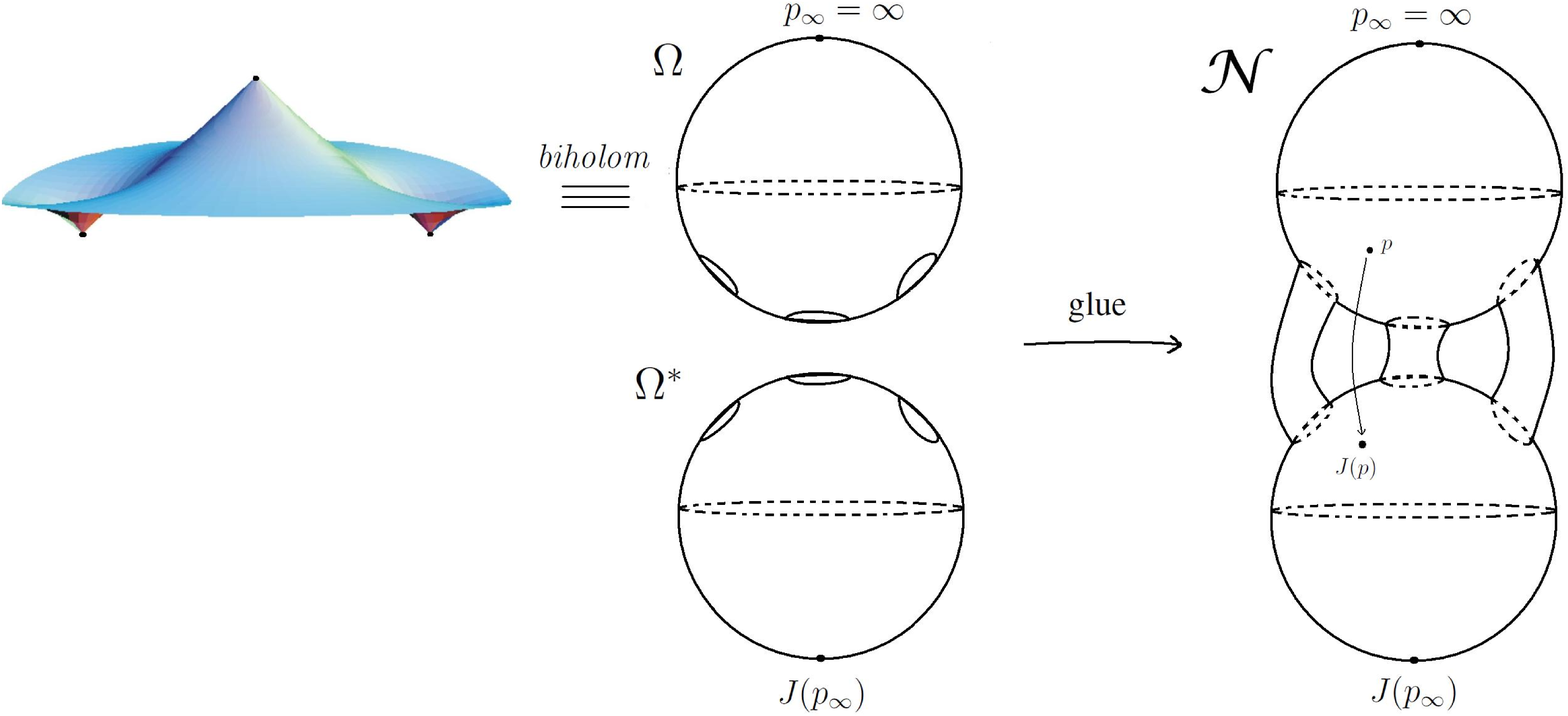}\caption{The double surface associated to a maximal surface with singularities}
\label{fig:mirror}
\end{figure}

This double surface is used in \cite{conelike} to give a characterization of complete
maximal surfaces with a finite number of singularities in terms of their Weierstrass data:

\begin{theorem}[Representation]
\label{th:representation}
Let $X:\Omega\to\l^3$ be a
conformal immersion of an entire maximal graph with $n+1$ conelike singularities, where
$\Omega=\c\setminus\cup_{j=1}^{n+1}\mbox{Int}(D_j),$ $D_j$ pairwise
disjoint closed discs.
Label $\ovr{\Nb}$ as the compactification of the double surface of
$\Omega.$ Then the Weierstrass data of $X,$ $(g,\phi_3)$, satisfy:
\begin{enumerate}[(i)]
\item
$g$ is a meromorphic function on $\overline{\Nb}$ of degree $n+1,$
$|g|<1$ on $\Omega$, and $g \circ J=\frac{1}{\overline{g}}$,
\item
$\phi_3$ is a holomorphic 1-form on
$\overline{\Nb}\setminus\{p_\infty,J(p_\infty)\},$ where $p_\infty=\infty
\in \Omega,$ with poles of order at most two at $p_\infty$ and
$J(p_\infty),$ and satisfying $J^*(\phi_3)=-\overline{\phi_3},$
\item
the zeros of $\phi_3$ in
$\overline{\Nb}\setminus\{p_\infty,J(p_\infty)\}$ coincide (with
the same multiplicity) with the zeros and poles of $g.$
\end{enumerate}

Conversely, let $\ovr{\Nb}$ be a compact genus $n$ Riemann
surface. Suppose that there exists an antiholomorphic involution
$J:\overline{\Nb} \rightarrow \overline{\Nb}$ such that the fixed
point set of $J$ consists of $n+1$ pairwise disjoint analytic
Jordan curves $\gamma_j,$ $j=0,1,\ldots,n,$ and that
$\overline{\Nb}\setminus\bigcup_{j=0}^n \gamma_j={\Omega_0} \cup
J({\Omega_0}),$ where $\overline{\Omega_0}$ is topologically equivalent (and so
conformally) to $\ovr{\c}$ minus a finite number of pairwise
disjoint open discs.

Then, for any $(g,\phi_3)$ satisfying $(i),$ $(ii)$ and $(iii)$
the map $ X:\overline{\Omega_0}\setminus\{p_\infty\} \rightarrow \l^3$
given by Equation \eqref{eq:repr} is well
defined and $S=X(\overline{\Omega_0}\setminus\{p_\infty\})$ is an entire maximal graph with conelike singularities corresponding
to the points $q_j:=X(\gamma_j),$ $j=0,$ $1,\ldots,n.$
\end{theorem}

%%%%%%%%%%%%%%%%%%%%%%%%%%%%%%%%%%%%%%%%%%%%%%%%%%%%%%%%%%%%%%%%%%%%%%%%%%%%%%%%%%%%%%%%%%%%%%%%%%%
\subsection{Divisors on a Riemann surface.}\label{sub:riemann}

An important part of our work in this paper deals with classical properties of divisors on compact Riemann surfaces. We recall here the notation and basics results that will be used in the sequel (see \cite{farkas} for more details).

Let $\Sigma$ be a Riemann surface. A (multiplicative) divisor on
$\Sigma$ is a formal symbol $\Db=p_1^{k_1}\cdot\ldots p_h^{k_h},$ where
$p_{k_j}\in\Sigma$ and $k_j\in\z.$ We can also write the divisor
$\Db$ as
$$\Db=\prod_{p\in\Sigma} p^{k_p},$$
where $k_p\neq 0$ only for finitely many.
We call $\div(\Sigma)$ to the multiplicative group of divisors on $\Sigma$. We can define an order in $\div(\Sigma)$, indeed, given
$\Db_1=\prod_{p\in\Sigma}p^{k_p^1}$ and $\Db_2=\prod_{p\in\Sigma}
p^{k_p^2} \in \div(\Sigma)$, we say that $\Db_1\geq \Db_2$ if $k_p^1\geq k_p^2$ for all $p\in\Sigma.$

The degree of the divisor $\Db$ is defined as the integer
$\deg(\Db)=\sum_{p\in\Sigma}k_{p}.$
$\Db\in \div(\Sigma)$ is an {\em integral} divisor if $k_p\geq 0$
for any $p\in\Sigma.$ We denote by $\div_k(\Sigma)$ the set of
integral divisors of degree $k.$\\

Let $f$ be a meromorphic function on $\Sigma.$ The associated
divisor of $f$ is defined as $(f)=\prod_{p\in\Sigma} p^{k_p},$
where for any zero (resp. pole) $p$ of $f$ of order $\alpha$ we
have $k_p=\alpha >0$ (resp. $k_p=-\alpha <0$), and $k_p=0$ in
other case. Likewise we define the associated divisor of a
meromorphic 1-form. Classical theory of Riemann surfaces give that both functions and $1$-forms are determined by their divisors up to a multiplying constant. Moreover, the degree of a meromorphic function  on a compact Riemann surface is $0$, whereas the associated divisor of a $1$-form has degree $2n-2$, where $n$ is the genus of the surface.

%%%%%%%%%%%%%%%%%%%%%%%%%%%%%%%%%%%%%%%%%%%%%%%%%%%%%%%%%%

\section{A first approach to the problem}\label{sec:first}

Let $G$ be an entire maximal graph with $n+1$ conelike singularities.
When $n=0$, Ecker \cite{ecker} characterized the Lorentzian catenoid (Figure \ref{fig:catrie}, left) as the unique entire maximal graph with
$1$ singular point, so we will assume from now on that $n\geq 1$.\\

As showed in Section \ref{sec:maximal}, the underlying conformal structure of a maximal graph is conformally equivalent to a circular domain $\Omega\subset\C$ with $n+1$ boundary components. Moreover, if we rotate the surface so that the end is horizontal, as a consequence of Theorem \ref{th:representation} the divisors of the Weierstrass data $(g,\phi_3)$ of $G$   must be of the form
\begin{equation}\label{eq:div}
(g)=\frac{D\cdot p_\infty}{D^\ast\cdot p_\infty^\ast},\qquad (\phi_3)=\frac{D\cdot D^\ast}{p_\infty\cdot p_\infty^\ast},
\end{equation}
where $p_\infty=\infty\in\overline\Omega$ is the end of the surface, $D\in\div_n(\overline\Omega)$, and $\,^\ast$ denotes the mirror involution.
Notice that the divisor $D$ determines uniquely the Weierstrass data $(g,\phi_3)$ up to replacing by $(e^{i\theta}\,g,A\,\phi_3)$, for any $\theta,A\in\R$.

Conversely, for any integral divisor $D$ of degree $n$ on $\overline\Omega$ such that there exist a meromorphic function $g$ and $1$-form $\phi_3$ satisfying \eqref{eq:div}, it is immediate to check that $(g,\phi_3)$ fulfill conditions $(i)$ to $(iii)$ in Theorem \ref{th:representation}. Thus  by means of Equation \eqref{eq:repr} we can obtain an entire maximal graph with $n+1$ conelike singularities, horizontal end, and conformal structure $\Omega$. Moreover, this graph is unique up to homotheties and vertical rotations.

The problem of finding out whether exists a pair $(g,\phi_3)$ satisfying \eqref{eq:div} for a given divisor $D$ is closely related with the Abel-Jacobi map of the corresponding compact Riemann surface $\overline{\Nb}$, $ \varphi: \div(\overline{\Nb}) \to \mathcal{J}(\overline{\Nb}) $, where $\mathcal{J}(\overline{\Nb})$ denotes the Jacobian bundle of  $\overline{\Nb}$ (see \cite{farkas} for its definition).
Abel Theorem states that $\Db\in\div(\overline{\Nb})$ is the divisor associated to a meromorphic function (resp. 1-form) on $\overline{\Nb}$ if and only if $\varphi(\Db)=0$ (resp. $\varphi(\Db)=T$, where $T\in \mathcal{J}(\overline{\Nb})$ is a fixed element in the Jacobian bundle).
Thus, in our case the divisors $D$ coming from Weierstrass data are precisely those satisfying:
$$\varphi(D)+\varphi(p_\infty) - \varphi(D^\ast) - \varphi(p_\infty^\ast)=0, \quad  \varphi(D) + \varphi(D^\ast) - \varphi(p_\infty) - \varphi(p_\infty^\ast) = T.$$

This set of divisors is deeply studied in \cite{conelike}, proving that the previous two equations are equivalent to
\begin{equation}\label{eq:spin}
2\varphi(D) - 2\varphi(p_\infty^\ast)=T.
\end{equation}
Before going into the properties of this set, let us fix some notation.
Let $\Omega$ be a $n$-connected circular domain and write  $\partial\Omega=\cup_{j=0}^{n} \gamma_{c_j}(r_j)$, with $\gamma_{c_j}(r_j)=\{z\in\C\;,\;|z-c_j|=r_j\}$.
Up to a Möbius transformations we can assume that $c_0=0,$ $r_0=1$ and $c_1\in\r^+$.
Thus, we can parameterize the space $\Tb_n$ of marked (i.e., with an ordering in the boundary components) $n$-connected circular domains (up to biholomorphisms) by their corresponding uplas $v=(c_1,r_1,\ldots,c_n,r_1,\ldots,r_n)\in\R^+\times\C^{n-1}\times(\R^+)^n,$ of centers and radii, with the convention $c_0=0$ and $r_0=1$.
By this identification, $\Tb_n$ can be considered as an open subset of $\r^+\times\C^{n-1}\times(\R^+)^n$, and therefore it inherits a natural analytic structure of manifold of dimension $3n-1$.
We label as $\Omega(v)$ the circular domain defined by $v\in\Tb_n$.
Now define the {\em spinorial bundle}
$$\Sb_n=\{ (v,D)\;:\; v\in\Tb_n,\; 2 \varphi_v(D)- 2 \varphi_v(p_\infty^\ast)   = T_v \},$$
where the subscript $v$ refers to the double surface of $\Omega(v)$, then

\begin{theorem}[\cite{conelike}] The spinorial bundle $\Sb_n$ defined above
is an analytical manifold of dimension $3n-1$ .
Moreover, the map
$$\nu:\Sb_n\to\Tb_n$$
$$\nu(v,D)=v$$
is a finitely sheeted covering.
\end{theorem}

Thus, the number of divisors $D\in\div_n(\overline{\Omega(v)})$ satisfying Equation \eqref{eq:spin}
is a universal constant that depends not on the conformal structure $\Omega(v)$, but only on the number of boundary components
(equivalently, the number of singularities of the maximal graph).
As explained above, each divisor corresponds to a unique congruence class of entire maximal graphs with $n+1$ singularities and conformal
structure $\Omega(v)$. Thus we have the following
\begin{corollary}\label{co:rec}
For each $n\in\N$ there exists a constant $C(n)\in\N$ such that, for any $n$-connected circular domain $\Omega$,
the number of non-congruent entire maximal graph with conformal structure biholomorphic to $\Omega$ is exactly $C(n)$.
\end{corollary}

\begin{remark}\label{re:comp}
Since the space $\Tb_n$ is simply-connected, it follows from Corollary \ref{co:rec}
that the number of connected components of $\Sb_n$ is $C(n)$. In particular, the number of connected components of the moduli space of entire maximal
graphs with $n+1$ singularities is also $C(n)$.
%See discussion in \cite{conelike} for a detailed explanation.

Indeed, label ${\mathcal{G}}_n$ as the space of {\em marked} entire maximal graph with horizontal end and $n+1$ singularities,
where a mark means an ordering $m=(q_0,\ldots,q_n)$ of the singular points of the graph. As we commented in Section \ref{sec:intro}, ${\mathcal{G}}_n$
can be endowed with a differentiable structure of manifold of dimension $3n+4$ with coordinates given by $(G,m)\mapsto (m,c)$,
being $c$ the logarithmic growth at the end.
On the other hand, we can consider the map
$$\epsilon: {\mathcal{G}}_n\to\Sb_n\times\L^3\times \S^1\times\R$$
$$\epsilon((G,m))=((v,D),q_0,g(1),h(1))$$
where, if $(g,\phi_3)$ denote the Weierstrass data of the graph, then
\begin{itemize}
\item $(v,D)\in\Sb_n$ is given by the conformal structure of $G$ (with the order in $v\in\Tb_n$ given by the order in $m$),
and the divisor $D$ defined as in Equation \eqref{eq:spin},
\item $q_0$ is the first singular point in $m$,
\item $h:=\frac{\phi_3}{dz}$ (here $z$ means the natural conformal parameter in $\Omega(v)\subset\C$, recall that $1\in\partial\Omega(v)$ for all $v\in\Tb_n$).
\end{itemize}
Then, it is clear from the above explanation that $\epsilon$ is bijective. Moreover, the induced topology in ${\mathcal{G}}_n$ by $\epsilon$ agree
with the one given by its before mentioned differentiable structure, as proved in \cite{conelike}. Thus, the number of connected components of
$\mathcal{G}_n$ is $C(n)$.
\end{remark}

%%%%%%%%%%%%%%%%%%%%%%%%%%%%%%%%%%%%%%%%%%%%%%%%%%%%%%%%%%

\section{Counting maximal graphs on a given circular domain}\label{sec:count}

As it was showed in the previous section, the number of maximal
graphs that share the same underlying conformal structure only depends on the number of boundary components of the conformal support.
Thus, in this section we will fix an specific circular domain and we
will find out how many non-congruent maximal graphs are defined on that
surface.\\

Let $n\in \n,$ and $a_1<a_2<\ldots <a_{2n+2}\in\r.$
Throughout this section, $\ovr{\Nb}_0$ will denote the
(hyperelliptic) compact genus $n$ Riemann surface associated to the
function $\sqrt{\prod_{j=1}^{2n+2}(z-a_j)},$ that is,
$$\ovr{\Nb}_0:=\{(z,w)\in\ovr{\c}^2\;:\;
w^2=\prod_{j=1}^{2n+2}(z-a_j)\}.$$
And we will also define $\Nb_0=\ovr{\Nb}_0\setminus\{z^{-1}(\infty)\}.$

The surface $\ovr{\Nb}_0$ can be realized as a two sheeted covering
of the Riemann sphere. Indeed, consider two copies of $\ovr{\c}.$
Following \cite{farkas}, we label these copies as sheet $I$ and
sheet $II.$ We "cut" each copy along curves joining $a_{2j+1}$
with $a_{2j+2},$ for any $j=1,\ldots,n.$ We assume that these cuts
does not intersect each others (see Figure \ref{fig:2copias}).
Each cut has two banks: a N-bank and a S-bank. We  recover  the
surface $\ovr\Nb_0$ by identifying  the N-bank (resp. S-bank) of a
cut in the sheet $I$ with the corresponding S-bank (resp. N-bank)
in the sheet $II.$

\begin{figure}[htbp]
\centering
\includegraphics[width=.5\textwidth]{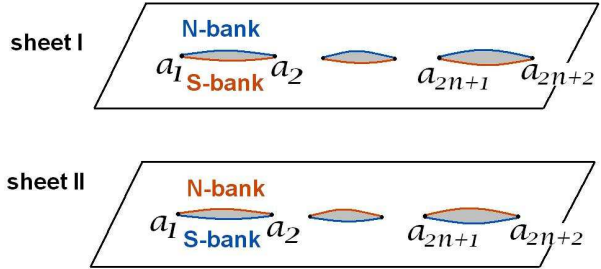}\label{fig:2copias}
\caption{A model for the Riemann surface $\Nb_0$.}
\end{figure}

We denote by $z,w:\ovr\nb_0\to\ovr{\c}$ the two canonical
projections, whose associated divisors are
$$(w)=\frac{a_1\cdot\ldots\cdot
a_{2n+2}}{(p_\infty)^{n+1}\cdot(p_\infty^*)^{n+1}}
\qquad\mbox{and}\qquad
(dz)=\frac{a_1\cdot\ldots\cdot a_{2n+2}}{(p_\infty)^{2}\cdot
(p_\infty^*)^{2}},$$ where $a_j\equiv(a_j,0)$ and $\{p_\infty,p_\infty^*\}=z^{-1}(\{\infty\}).$
We will label $p_\infty$ as the one where the coefficient of degree $-(n+1)$ of the Laurent series of $w$ is $-1$.\\

Finally we define $J_0:\ovr{\Nb}_0\to\ovr{\Nb}_0$ as the antiholomorphic involution
given by $J_0(z,w)=(\ovr{z},-\ovr{w}).$ The fixed points of $J_0$ are
the Jordan curves
$\gamma_j=\{(z,w)\in\ovr{\Nb}_0\;:\;z\in[a_{2j-1},a_{2j}]\},$
$j=1,\ldots,n+1.$ Moreover, $\Nb_0\setminus\cup_{j=1}^{n+1}\gamma_j$
has two connected components, each one
of them corresponding to a single-valued branch of $w$,
and biholomorphic to a $n$-connected circular domain.

\begin{definition}\label{def:omega}
Let $n\in \n,$  and $a_1<a_2<\ldots <a_{2n+2}\in\r.$ Consider the above defined compact Riemann surface
$$\ovr{\Nb}_0:=\{(z,w)\in\ovr{\c}^2\;:\;
w^2=\prod_{j=1}^{2n+2}(z-a_j)\},$$
with the antiholomorphic involution $J_0(z,w)=(\ovr{z},-\ovr{w})$. Label $\Delta$ as the set of fixed points of $J_0$.
We will define $\bar\Omega_0$ as the closure of the connected component of $\ovr\Nb_0\setminus\Delta$
containing $p_\infty$, and $\Omega_0$ will denote the circular domain $\Omega_0:=\bar\Omega_0\setminus\{p_\infty\}$.

\end{definition}

%%%%%%%%%%%%%%%%%%%%%%%%%%%%%%%%%%%%%%%%%%%%%%%%%%%%%%%%%%%%%%%%%%%%%%%%%%%%%%%%%%%%%%
\begin{proposition}\label{pro:caract}
Let $(g,\phi_3)$ be Weierstrass data on $\Omega_0$ of an entire maximal graph with $n+1$ singularities and
horizontal end. Then there exists $n+1$ distinct points
$\{b_1,\ldots,b_{n+1}\}\subset\{a_1,\ldots,a_{2n+2}\}$, such that

\begin{equation}\label{eq:datos}
g=e^{i\theta}\frac{w+P(z)}{w-P(z)} \quad \mbox{and} \quad
\phi_3=A \big( \frac{w}{P(z)} - \frac{P(z)}{w} \big) dz,
\end{equation}
where $P(z)=\prod_{j=1}^{n+1}(z-b_j)$, $\theta\in\r$, and
$A\in\r^\ast$.
\end{proposition}

\begin{proof}
By Theorem \ref{th:representation}, the associated divisors to
$(g,\phi_3)$ are given by
\begin{equation}\label{eq:divisor}
(g)=\frac{D\cdot p_\infty}{J(D)\cdot
p_\infty^\ast}\quad\mbox{and}\quad (\phi_3)=\frac{D\cdot
J(D)}{p_\infty\cdot  p_\infty^\ast}
\end{equation}
where $D\in Div_n(\overline{\Omega}_0).$ Here,
$p_\infty$ denotes the point in $\ovr{\Omega}_0\cap z^{-1}(\infty)$, and $p_\infty^\ast=J(p_\infty)$.

We will denote by $F:\ovr{\Nb}_0\to\ovr{\Nb}_0$ the holomorphic
involution given by $F(z,w)=({z},-{w})$.

\begin{assertion}\label{claim1}
In the above conditions there exist $n+1$ distint points $\{b_1,\ldots,b_{n+1}\}\subset\{a_1,\ldots,a_{2n+2}\}$, such that $g=\frac{G_1}{G_2}$ for two meromorphic functions $G_1$, $G_2$ on
$\overline\Nb_0$ satisfying
\item \begin{enumerate}[a)]
\item $(G_1)\geq\displaystyle{\frac{p_\infty^*}{b_1\cdot\ldots\cdot b_{n+1}}}$
\item$(G_2)\geq\displaystyle{\frac{p_\infty}{b_1\cdot\ldots\cdot b_{n+1}}}$
\end{enumerate}
\end{assertion}

Since $g$ has degree $n+1$ and $\overline\Nb_0$ is hyperelliptic, the two meromorphic
functions $g$ and $z$ satisfy a relation $P(g,z)=0,$ where $P$ is
a polynomial in two variables with algebraic degree two in the
first one and $n+1$ in the second (see \cite{farkas}). We can rewrite this
relation as $P_2(z)g^2+P_1(z)g+P_0(z)=0,$ with $P_i$ polynomials
whose maximum algebraic degree is $n+1.$ Solving this equation we
obtain
$$g=\frac{-P_1\pm \sqrt{P_1^2-4P_0P_2}}{2P_2}.$$

Consider the meromorphic function
$f=\sqrt{P_1^2-4P_0P_2}=\pm(2gP_2+P_1).$ Let us check that $f=c
w,$ for some constant $c\in\r^*.$
Indeed, any meromorphic function on the hyperelliptic surface $\ovr{\Nb}_0$ can be expressed
as $f=R_1(z)+R_2(z)w,$ with $R_i$ rational functions (see \cite{farkas}).
In our case, $f^2$ is a polynomial function in $z$, and so it
follows that either $R_1=0$ or $R_2=0.$ The last case would imply
that $g$ is a rational function of $z,$ which is impossible from
Equation (\ref{eq:divisor}) so $f=R_2(z)w.$ Now observe that $f$
has poles only at $p_\infty$ and $p_\infty^\ast$ with order at
most $n+1$, which implies that $f/w$ is a holomorphic function on
$\ovr{\Nb}_0,$ and therefore constant. Thus, $f=cw$ for some
$c\in\r^\ast$.
Up to replace $P_i$ by $\pm cP_i,$ $i=1,2,$  we can suppose that
$$g=\frac{ P_1+w }{2 P_2}.$$ We will also assume that the leading
coefficient of $P_1$ is one. Since $P_1$ and $P_2$ are meromorphic functions of degree $\leq 2(n+1)$ that only depend on $z$,
 it is not hard to realize that \eqref{eq:divisor}
implies that $$(P_1+w)=\frac{D\cdot
E}{(p_\infty)^{n-1}\cdot(p_\infty^*)^{n+1}}
\quad\mbox{and}\quad (P_2)=\frac{J(D)\cdot
E}{(p_\infty)^n\cdot(p_\infty^*)^n},$$
where $E:=F(J(D)\in Div_n(\overline\Omega_0)$. Thus, the
meromorphic function $$h=\frac{P_2(P_1+w)}{w\prod_{e\in E}
(z-z(e))}\frac{dz}{\phi_3}$$
satisfies that $(h)=\frac{E\cdot p_\infty}{F(E)\cdot
p_\infty^*}=(\frac{1}{h}\circ F),$ and therefore up to a multiplying constant $h\circ F=1/h.$
On the other hand, $\mbox{deg}(h)=n+1$, and reasoning as before we can deduce that $h=(\hat{P}_1(z)+w)/\hat{P}_2(z),$
for some $\hat{P}_i(z)$ polynomial functions in $z$ with algebraic
degree less than or equal to $n+1.$ Since $h\circ F=1/{h},$ we
infer that $w^2=\hat{P}_1^2-\hat{P}_2^2$ and so, setting
$S=-\hat{P}_1-\hat{P}_2$ we can write $h=(S-w)/(S+w).$

Looking at the divisor of $h$ is immediate to realize that there exists an
integral divisor $B$ with $\deg B=n+1$ such that:
$$(S-w)=\frac{E\cdot B}{p_\infty^{n} \cdot(p_\infty^*)^{n+1}}
\quad \textrm{and} \quad (S+w)=\frac{F(E)\cdot
B}{p_\infty^{n+1}\cdot (p_\infty^*)^{n}}.$$ Since points in $B$
are zeros of both $S+w$ and $S-w$, they must be $n+1$ distinct
(recall that $w$ only has simple zeroes) points of $\{a_1\dots
a_{2n+2}\}.$
Setting $G_1=\displaystyle{\frac{{P}_1+w}{S-w}}$ and $G_2=\displaystyle{\frac{2{P}_2}{S-w}}$ the claim is proved.\\

\begin{assertion}\label{claim2} Up to multiplicative constants, the functions $G_1$ and $G_2$ in Claim \ref{claim1} are given by
$G_1= \displaystyle{\frac{w}{P(z)}+1} $  {and}  $G_2=\displaystyle{ \frac{w}{P(z)}-1},$ being
$P(z)=\prod_{j=1}^n(z-b_j)$.
\end{assertion}

Call $B$ to the integral divisor given by $B=b_1\cdot\ldots\cdot
b_{n+1}$.
By Riemann-Roch Theorem,  the
dimension of the linear space of meromorphic functions on
$\overline{\Nb}_0$ satisfying condition $a)$ (resp. $b)$) in
 Claim \ref{claim1} is $1+d$ where $d$ is the dimension of the linear space
of meromorphic 1-forms $\nu$ on $\overline{\Nb}_0$ satisfying
$(\nu)\geq \frac{B}{p_\infty^*}$ (resp. $(\nu)\geq
\frac{B}{p_\infty}$). Let us see that $d=0.$

Indeed, observe first that by the residues theorem, both spaces agree with the space $L(B)$ of holomorphic $1$-forms $\nu$ with $(\nu)\geq B$.
But since $\{ \displaystyle{\frac{dz}{w}}, z\displaystyle{\frac{dz}{w}},\ldots,
z^{n-1}\displaystyle{\frac{dz}{w}}\}$ is a basis for the space of holomorphic
1-forms on $\overline{\Nb}_0,$ any $\nu\in L(B)$ must be of the form
 $\nu=P(z)\frac{dz}{w},$ where $P$ is a polynomial with
algebraic degree less than $n.$ Thus, if a Weierstrass point
$a_{j_0}$ is a zero of $\nu$ then its order is at least two. It
follows that the number of zeroes of the holomorphic 1-form $\nu$
is at least $2(n+1)$ which is impossible because $\ovr{\Nb}_0$ has
genus $n.$

Therefore the dimension of the linear space of meromorphic
functions satisfying condition $ a)$ (resp. $ b)$) in the
Claim \ref{claim1} is $1.$ It is easy to show that the function
$\frac{w}{\prod_{j=1}^{n}(z-b_j)}+1$  (resp.
$\frac{w}{\prod_{j=1}^{n}(z-b_j)}-1$) is a basis for this space,
so Claim \ref{claim2} is proved.

\vspace*{.5cm}

As a consequence of the previous claims, we can write: $$g= \frac{G_1}{G_2}=
c\,\frac{w+P(z)}{w-P(z)},$$ for a suitable constant $c\in\c^\ast$. As $g\circ J=1/\overline{g}$ we infer
that $c=e^{i\theta}$ for some $\theta\in\R$.

To finish observe that the divisor of $\phi_3$ coincides with the
divisor for the 1-form $ \big( \frac{w}{P(z)} - \frac{P(z)}{w}
\big) dz,$ and as a consequence $$\phi_3=A\, \big(
\frac{w}{P(z)} - \frac{P(z)}{w} \big) dz,$$ since
$J^*(\phi_3)=-\overline{\phi_3}$  we get $A\in\r.$ This concludes
the proof.
\end{proof}

%%%%%%%%%%%%%%%%%%%%%%%%%%%%%%%%%%%%%%%%%%%%%%%%%%%%%%%%%%%%%%%%%%%%%%%%%%%%%%%%%%

To finish the classification  of the entire maximal graphs on the given circular domain $\Omega_0$ we need to
find out when the pair given by \eqref{eq:datos} are actually Weierstrass data. This is done in the following proposition.
Figure \ref{fig:alineados} shows two examples of the surfaces given by these Weierstrass representation.

\begin{proposition}\label{pro:b}
Choose $b_1<b_2<\ldots< b_{n+1}$ points in
$\{a_1,\ldots,a_{2n+2}\}$, and define
$P(z)=\prod_{j=1}^{n+1}(z-b_j)$.

Then the pair $(g,\phi_3)$ given by Equation \eqref{eq:datos} are
Weierstrass data on $\Omega_0$ of an entire maximal graph with $n+1$
singularities if and only if $b_j\in\{a_{2j-1},a_{2j}\}$ for all
$j=1,\ldots,n+1$.
\end{proposition}

\begin{proof}

We just have to check the conditions stated in Theorem
\ref{th:representation}. Recall that $J(z,w)=(\bar{z},-\bar{w})$,
and define $Q(z)=w^2/P(z)=\prod_{j=1}^{n+1}(z-c_j)$. For
simplicity, we will assume that $\theta=0$ and $A=1$.

Conditions $(ii)$ and $(iii)$ are straightforward for all the
possible values of $b_1,\ldots,b_{n+1}$. Let us show when $(i)$ is
accomplished.

First, notice that $g^{-1}(1)=\{b_1,\ldots,b_{n+1}\}$. In
particular, $\deg(g)=n+1$. In particular, in order to be $g$ the Gauss map of a maximal
surface with conelike singularities, any connected component in
$\partial\Omega_0$ must have exactly one point with $g=1$, and so $b_j\in\{a_{2j-1},a_{2j}\}$ for every $j=1,\ldots, n+1$.

Conversely, assume that $b_j\in\{a_{2j-1},a_{2j}\}$, $j=1,\ldots, n+1$, and let us show that
$g$ has no critical points on
$\partial\Omega_0\equiv\cup_{j=1}^{n+1}[a_{2j-1},a_{2j}]$. After some
computations one easily gets that
$$dg=\frac{QdP-PdQ}{w(Q+P-2w)}.$$ Thus for critical points in
$\Nb_0=\ovr{\Nb}_0\setminus\{z^{-1}(\infty)\}$ we have $QdP=PdQ$, or
equivalently, $$\sum_{j=1}^{n+1} \frac{1}{z-b_j} =
\sum_{j=1}^{n+1} \frac{1}{z-c_j}.$$

If we assume that $b_j\in\{a_{2j-1},a_{2j}\}$ for all
$j=1,\ldots,n+1,$ and we have a point $p_0\in
[a_{2j_0-1},a_{2j_0}]\subset \partial\Omega_0$, with
$a_{2j_0-1}=b_{j_0}$ and $a_{2j_0}=c_{j_0}$ (the case
$a_{2j_0-1}=c_{j_0}$ and $a_{2j_0}=b_{j_0}$ is similar) then we
have that $$ \frac{1}{z(p_0)-c_j} < \frac{1}{z(p_0)-b_{j+1}}<0,
\quad j=1,\ldots n+1,$$
(here we use the convention $b_{n+2}=b_1$), and this gives that $p_0$ cannot be a critical point
of $g$.

To finish just notice that $g\circ J =1/\bar{g}$ and therefore
$|g|=1$ on the $n+1$ connected components of $\partial\Omega_0$. Since
$g$ is injective on each one of these curves, and $\deg(g)=n+1$,
then $|g|\neq 1$ on $\Nb_0\setminus\partial\Omega_0$. Taking into
account that $g(p_\infty)=0$ we have that $|g|<1$ on $\Omega_0$.
\end{proof}

%%%%%%%%%%%%%%%%%%%%%%%%%%%%%%%%%%%%%%%%%%%%

\begin{definition}\label{def:wei}
Let $\Omega_0$ the circular domain given in Definition \ref{def:omega} for some real numbers $a_1<\ldots<a_{2n+2}$.
For each subset $\tau=\{b_1,\ldots,b_{n+1}\}\subset\{a_1,\ldots,a_{2n+2}\}$ with $b_j\in\{a_{2j-1},a_{2j}\}$, $j=1,\ldots,n+1$, we will define the
 $G_\tau$ as the entire maximal graph with $n+1$ singularities with Weierstrass data $(g_\tau,\phi_3^\tau)$ on $\Omega_0$ given by
$$ g_\tau= = \frac{w+P(z)}{w-P(z)} \quad \mbox{and} \quad
\phi_3^\tau=  \big( \frac{w}{P(z)} - \frac{P(z)}{w} \big) dz,
$$
where $P(z)=\prod_{j=1}^{n+1}(z-b_j)$.
\end{definition}

\begin{figure}[htbp]
\centering
\includegraphics[width=.8\textwidth]{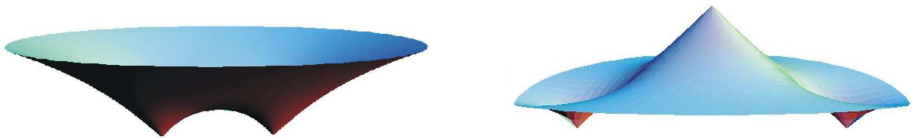}\label{fig:alineados}\caption{Two examples of the surfaces obtained for $n=1$ and $n=2$}
\end{figure}

%%%%%%%%%%%%%%%%%%%%%%%%%%%%%%%%%%%%%%%%%%%%%%%%%%%%%%%%%%%%%%%%%%%%%%%%%%%%%%%%%

\begin{theorem}
Let $\Omega_0$ be the $n$-connected circular domain given in Definition \ref{def:omega}.
Then the number of non-congruent entire maximal graphs whose underlying conformal structure is $\Omega_0$ is exactly $2^{n}.$
\end{theorem}

\begin{proof}
From Propositions \ref{pro:caract} and \ref{pro:b} we know that any maximal graph $G$ with horizontal end defined on $\Omega_0$
have Weierstrass data $(g=e^{i\theta} g_\tau,\phi_3=A\phi_3^\tau),$ where $\theta\in\R$, $A\in\R^\ast$ and $(g_\tau,\phi_3^\tau)$
are given by Definition \ref{def:wei}.

Observe that replacing the set $\tau$ by its complementary $\{a_1,\ldots,a_{2n+2}\}\setminus \tau$
gives congruent surfaces (more specifically, $(g,\phi_3)$ are transform into $(-g,-\phi_3)$).
So, we can assume without loss of generality that $b_1=a_1$.
To avoid congruences, we will also normalize so that $g(a_1)=h(a_1)=1$,
where $h=\displaystyle{\frac{\phi_3}{\frac{dz}{w}}}$. Looking at the expressions for $g$ and $\phi_3$ this means that $\theta=0,$ $A=1$.
Thus, the number of non-congruent maximal graphs defined on $\Omega_0$ is the number of possible choices of
$b_j\in\{a_{2j-1},a_{2j+1}\},$ $j=2,\ldots, n+1$, which is $2^n$.
\end{proof}

Taking into account our previous discussion in Section \ref{sec:first}, we can conclude that:

\begin{theorem}\label{th:main}
The number of non-congruent entire maximal graphs with the same conformal structure is $2^n$, where $n+1$ is the number of (conelike) singularities.

Equivalently, the number of connected components of the space ${\mathcal{G}}_n$ of entire marked maximal graphs with $n+1$ singularities and
horizontal end is $2^n$.
\end{theorem}

%%%%%%%%%%%%%%%%%%%%%%%%%%%%%%%%%%%%%%%%%%%%%%%%%%%%%%%%%%%%%%%%%%%%%%%%%%%%%%

\section{Maximal graphs with coplanar singularities}\label{sec:cons}

We will prove now that the surfaces constructed in the previous section are characterized
by the property of having all its singularities on a plane orthogonal to the limit  normal vector at infinity. In particular, for $n=1$, surfaces obtained in Section \ref{sec:count} describe  the whole moduli space of the entire maximal graphs with two singular points.

\begin{theorem}\label{th:alineados}
Let $G\subset\l^3$ be an entire maximal graph with $n+1$
conelike singularities. Then $G$ has all its singularities lying on a timelike plane in $\L^3$ orthogonal (in the Lorentzian sense)
to the normal vector at the end if and only if
$G$ is congruent to one of the examples given in Definition \ref{def:wei}.
\end{theorem}

\begin{proof}
Assume that $G$ has all its singularities in an orthogonal plane to the normal vector at the end.
Up to a rigid motion in $\L^3$ we can assume that the end is horizontal and the singularities lie in the plane $\{x_1=0\}$.
Let $X:\Omega\to G\subset\L^3$ a conformal reparameterization of $G$.
By the uniqueness result in \cite{klyachin} (see also \cite{conelike} Remark 2.5), the surface is symmetric with respect to the plane
$\{x_1=0\}.$ This symmetry induces an antiholomorphic involution
$T:\overline{\Omega}\to \overline{\Omega}$ leaving $\partial\Omega$
globally fixed. It follows that $T$ extends to an antiholomorphic
involution $T:\overline\Nb\to \overline\Nb$, where $\Nb$ is the mirror surface, by putting $T\circ
J=J\circ T$ ($J$ is the mirror involution). Moreover, if $(g,\phi_3)$ are the Weierstrass data of the immersion, $g\circ T=\overline{g}$ and
$T^*(\phi_3)=\ovr{\phi_3}.$
It is straightforward that $T$ must have exactly two fixed points
on every connected component of the circular domain $\partial \overline{\Omega}.$ We call
these points $p_1,\ldots,p_{2n+2}.$ Observe that the end $p_{\infty}\in\overline\Omega$ is
also fixed by $T.$

Consider the holomorphic involution $F=J\circ T,$ whose fixed points are exactly $p_1,\ldots, p_{2n+2}.$ Therefore, $\ovr\Nb$ is a compact genus $n$ Riemann surface with $2n+2$ fixed points, this means that $\ovr\Nb$ is hyperelliptic with Weiersrtass points $p_1,\ldots,p_{2n+2}$ (see \cite{farkas}),
$$\overline\Nb  \equiv
\{(z,w)\in\overline{\c}^2\;:\;w^2=\prod_{i=1}^{2n+2}(z-a_j)\},$$
where $(a_j,0)$ corresponds to $p_j$ for any $j$ (and so $a_j\neq
a_k$ for $k\neq j$). With this identification we  have
$F(z,w)=(z,-w).$ Up to a Möbius transformation we can suppose that
$z(p_{2n+1})=1,$ $z(p_{2n+2})=-1,$ and $z(p_\infty)=\infty.$

In what follows we will identify $a_j=(a_j,0)\in \overline\nb.$ To
prove $a_j\in\r$ notice that the divisor associated to the
meromorphic 1-form $d(\overline{z\circ J})$ coincides with the one
of $dz$ and therefore $\overline{z\circ J}=k\,z+\lambda,$ for some
$k,\lambda\in\r.$ Since $a_{2n+1}=1$ and $a_{2n+2}=-1$ are fixed
by $J$ it follows that $\overline{z\circ J}=z$ which implies that
$a_j\in\r.$
Moreover, since $w^2\circ J=\overline{w}^2,$ then $w\circ J=\pm
\overline{w}.$ Taking into account that $J$ interchanges the two
points with $z=\infty,$ namely $p_\infty$ and
$p_\infty^*=J(p_\infty),$ then $w\circ J=- \overline{w}.$
Therefore $J(z,w)=(\ovr{z},-\ovr{w})$ and
$T(z,w)=(\ovr{z},\ovr{w}).$ In particular, $\Omega$ agrees with the circular domain $\Omega_0$ defined in Definition \ref{def:omega} and by Propositions
\ref{pro:caract} and \ref{pro:b} we are done. \\

Conversely, let $G_\tau$ one of the graphs defined in Definition \ref{def:wei}.
Consider the involution $T(z,w) = (\bar{z},\bar{w})$ on $\ovr\nb_0$ that fix globally any component of $\partial\Omega_0$.
Moreover, $g_\tau\circ T=\bar{g}_\tau$ and $T^\ast(\phi_3^\tau)=\bar{\phi}_3^\tau$, thus, $T$ induces an isometry
on the resulting surface, namely $I(x_1, x_2, x_3) = (-x_1, x_2, x_3).$ Since $\{a_1,\ldots,a_{2n+2}\}$ are fixed by $T$ it
follows that all the singularities lie in the plane $\{x_1 = 0\}$.
\end{proof}

%%%%%%%%%%%%%%%%%%%%%%%%%%%%%%%%

\end{document}